\documentclass[a4paper,10pt]{amsart}

\usepackage[cp850]{inputenc}
\usepackage{marvosym}
\usepackage{wasysym}
\usepackage{latexsym}
\usepackage{amsfonts}
\usepackage{amssymb}
\usepackage{amsmath}
\usepackage{mathrsfs}
\usepackage{color}
\usepackage[colorlinks]{hyperref}
\usepackage{fancyhdr}
 
\newtheorem{theorem}{Theorem}[section]
\newtheorem{lemma}[theorem]{Lemma}

\newtheorem{corollary}[theorem]{Corollary}

  \newtheorem{example}[theorem]{Example}

  \newtheorem{remark}[theorem]{Remark}

\newcommand{\C}{\mathbb{C}}

\newcommand{\R}{\mathbb{R}}

\newcommand{\N}{\mathbb{N}}

\newcommand{\1}{1\!\! 1}

\newcommand{\norm}[1]{\left\lVert#1\right\rVert}

\title{Stability for weighted composition $C_0$-semigroups on Lebesgue and Sobolev spaces}
\author{Javier Aroza}
\address{Javier Aroza, Departament de Matem\`atiques - Analisi Matem\`atica, Universitat Jaume I de Castell\'o, 
E-12071 Castell\'o de la Plana}
 \email{aroza@uji.es}
\author{Elisabetta Mangino}
\address{Elisabetta  Mangino, Dipartimento di Matematica e Fisica "E. De Giorgi", Universit\`a del Salento, I-73100 Lecce, Italy}
\email{elisabetta.mangino@unisalento.it}

\thanks{This study is supported in part by Generalitat Valenciana, Project
PROMETEO/2008/101. The first author acknowledges the support of grant
MHE2011-00201 from Ministerio de Educaci\'on. The authors thank Thomas Kalmes for discussions on  the topic of this paper.}

\begin{document}

\begin{abstract} Stability of weighted composition strongly continuous  semigroups acting on Lebesgue and Sobolev  spaces is studied,  without the use of spectral conditions on the generator of the semigroup. Applications to the generalized von Foerster - Lasota semigroup and  a comparison with hypercylicity conditions are presented.

\end{abstract}

\maketitle

Let $\Omega\subseteq\R^N$ be an open set 
and let $\varphi:[0,\infty[\times \Omega \rightarrow \Omega$ be  a semiflow, i.e.  a continuous function such that  $\varphi(0,\cdot)=id_\Omega$, $\varphi(t,\cdot)\circ \varphi(s,\cdot )=\varphi(t +s,\cdot)$ for all $t,s\geq 0$ and such that $\varphi(t, \cdot)$ is injective for all $t\geq 0$. For example, $\varphi$ could be the solution  of the initial value problem
\begin{equation}\label{flow} x' =F(x), x(0)=x_0\end{equation}
on $\Omega$, where $F:\Omega\rightarrow \R$ is a locally Lipschitz continuous vector field over  $\Omega$ and $x_0\in \Omega$.

Moreover let $h:\Omega\rightarrow \R$ be a continuous function and 
define for every $t\geq 0$:
\[h_t:\Omega\rightarrow (0,\infty), x\mapsto\exp(\int_0^t h(\varphi(s,x))\,ds).\] 

Denoting by  $\lambda$ be the $N$-dimensional Lebesgue measure,   let $\rho: \Omega \rightarrow (0,+\infty)$ be a  $\lambda$-measurable locally integrable function and consider, for every $1\leq p<\infty$, the space 
$$
L^p_\rho(\Omega)=\{ u: \Omega\rightarrow \R\,\mid\, u  \ \lambda-\mbox{measurable and } \Vert u\Vert_p<\infty\},$$
where $\Vert u\Vert_p=\left(\int_{\Omega}\vert u(t)\vert^p\rho(t) dt\right)^{\frac 1 p}$. Finally set

\begin{equation} (T(t)f)(x):=h_t(x)f(\varphi(t,x)), \qquad f\in L^p_\rho(\Omega), \ x\in \Omega,\  t\geq0\label{semi}\end{equation}

If $\varphi$ is continuously differentiable, in  \cite[Proposition 4.12, Theorem 4.7]{kalmes2007hypercyclic},
 it has been proved that  $\mathcal T= (T(t))_{t\geq 0}$ is a $C_0$-semigroup in $L_\rho^p(\Omega)$ if and only if
 there exist $M\geq 1$ and $\omega\in\R$ such that, for every $t\geq 0$,  the following inequality holds 
  \[|h_t|^p\rho \leq Me^{\omega t}\rho(\varphi(t,\cdot  ))|{\rm det} D\varphi(t, \cdot)|\mbox{, } \qquad \lambda\mbox{-a.e.\ on }\Omega,\]
  where $D\varphi(t, \cdot)$ is the Jacobian of $\varphi(t, \cdot)$, or equivalently
\begin{equation*}
  \chi_{\varphi(t, \Omega)}\frac{ h_t^p(\varphi(-t, \cdot)) \rho (\varphi(-t, \cdot))|{\rm det} D\varphi(-t, \cdot)|}{\rho} \leq M e^{t\omega}\qquad   \lambda-\mbox{a.e.},
\end{equation*}
  where $\varphi(-t, \cdot)$ is the inverse function of $\varphi(t, \cdot)$ and $D\varphi(-t, \cdot)$ is its Jacobian. 
  In this case $\rho$ is said to be a $p$-admissible weight and for every $t\geq 0$:
	  \begin{equation}\label{normT} ||T(t)||=\left\Vert \frac{\rho_{t,p}}{\rho}\right\Vert_\infty, \end{equation}
  where
  \[ \rho_{t,p}=\chi_{\varphi(t, \Omega)} |h_t(\varphi(-t, \cdot))|^p \rho (\varphi(-t, \cdot))|{\rm det} D\varphi(-t, \cdot)|.\]

  Our aim is to characterize the  stability of the $C_0$-semigroup $\mathcal T$ on $L^p_\rho(\Omega)$, namely when
  $$ \forall f \in L^p_\rho(\Omega) \qquad \lim_{t\rightarrow +\infty}\|T_t f\|_p=0.$$ 

The classical approach to stability is to analyze the spectrum of the generator of the semigroup (see e.g. \cite{engel_nagel2000one-parameter}). We will instead characterize stability in terms   of the functions $\rho_{t,p}$. 

After studying general conditions in  the multidimensional case,  we will focus on the one-dimensional case,  obtaining  explicit characterizations when the semiflow is associated to a function  $F$ as in \eqref{flow}.   Similar characterizations for weighted composition  semigroups acting on Sobolev spaces $W^{1,p}(a,b)$ are considered and finally, applications to the semigroups associated with the  (generalized)  linear von Foerster-Lasota equations are investigated. In the appendix,  we will compare the stability conditions with the hypercyclicity conditions for the same semigroups   that were obtained in \cite{aroza_kalmes_mangino2013chaotic, kalmes2007hypercyclic, kalmes2009hypercyclic}.

\section{The multidimensional case}
By sake of completeness,  we state an essential known lemma about convergence with respect to the $w^*$-topology in $L^\infty (\Omega, \mu)$  (see e.g. \cite{brezis2011functional}):
  
  \begin{lemma}\label{w*} Let $(\Omega, \mu)$ be  a $\sigma$-finite measure space and  let $\psi:[0,+\infty[ \rightarrow L^\infty(\Omega,\mu)$ be a locally bounded function.
  Then $w^*$-$\lim_{t\to\infty} \psi(t)=0$ if and only if 
  \begin{enumerate}
  \item$ \sup_{t\geq 0} ||\psi(t)||_\infty<+\infty$;
  \item for every $Q\subseteq\Omega$ with $\mu(Q)<+\infty$, it holds
\[ \lim_{t\to +\infty} \int_{Q}\psi(t) d\mu =0.\]
\end{enumerate}
  \end{lemma}
  
\begin{theorem}\label{characterization stability_1}
Let $\Omega$ be an open subset of $\R^N$, $\varphi$ a continuously differentiable semiflow on $\Omega$, $h\in C(\Omega)$, $\rho$ a $p$-admissible function, and let $\mathcal{T}$ be the semigroup on $L^p_\rho (\Omega)$ defined by \eqref{semi}. 
Then  the  following conditions are equivalent:
\begin{itemize}
	\item[(i)] $\mathcal{T}$ is stable on $L^p_\rho (\Omega)$,
	\item[(ii)] It holds:
\begin{enumerate}
\item[(1)] $\mathcal{T}$ is bounded.
\item[(2)] for every bounded interval $Q\subseteq\Omega$
\[ \lim_{t\to+\infty} \int_{Q}\rho_{t,p}(x)dx =0.\]
\end{enumerate}
\end{itemize}
\end{theorem}
									
\begin{proof}
Observe that, by \eqref{normT}, $\mathcal{T}$ is bounded if and only if $\sup_{t\geq 0} \left\Vert \frac{\rho_{t, p}}{\rho}\right\Vert_\infty <+\infty$.

Then $\mathcal{T}$ is stable on $L^p _{\rho}(\Omega)$ if and only if 

\[ \lim_{t\to+\infty} \int_{\Omega} \left\vert h_t(x)f(\varphi(t,x))\right\vert^p \rho(x)dx=0, \qquad \forall f\in L^p_\rho(\Omega),\]
or, equivalently, with a change of variable (see also the argument in  the proof of
 \cite[Proposition 3.11]{kalmes2007hypercyclic}):
\[ \lim_{t\to+\infty} \int_{\Omega} |f(y)|^p \rho_{t, p}(y)dy=0, \qquad \forall f\in L^p_\rho(\Omega)\]
that is
\[ \lim_{t\to+\infty} \int_{\Omega} |g(y)| \rho_{t, p}(y) dy=0, \qquad \forall g\in L^1_\rho(\Omega).\]

By applying the previous to $ g^+=g\vee 0$ and $g^-=g\wedge 0$, we get that it is equivalent to 
\[ \lim_{t\to+\infty} \int_{\Omega} g(x) \frac{\rho_{t, p}(x)}{ \rho(x)}\rho(x)dx=0, \qquad \forall g\in L^1_\rho(\Omega).\] namely,$ \frac{\rho_{t, p}(x)}{ \rho(x)}\rightarrow 0$ with respect to the $w^*$-topology on $L^\infty(\Omega)$ induced  by $L^1_\rho(\Omega)$. Since the function $t\mapsto \rho_{t,p} \in L^\infty(\Omega)$ is locally bounded,  the last assertion  is equivalent to (ii) by Lemma \ref{w*}.
\end{proof}

\begin{remark}
{\rm We could also consider the complex valued function space  $L^p_\rho(\Omega, \C)$, $h\in C(\Omega, \C)$ and 
define for every $f\in L^p_\rho(\Omega, \C)$ and every $t\geq 0$ 
\[T_{h}^{\C}(t)(f)(x)= \exp \left( \int_0^t h(\varphi(s,x))\right) f(\varphi(x,t)).\]
By the quoted results in  \cite{kalmes2007hypercyclic},  if $\rho$ is admissible, then  $ \mathcal T^{\C}_{h}=(T^{\C}_{h}(t))_{t\geq 0}$ is a strongly continuous semigroup on $L^p_\rho(\Omega, \C)$ .
Consider, for every $f\in L^p_\rho(\Omega, \R)$ and $t\geq 0$
\[ T_{{\rm Re}h}^{\R}(t)(f)(x)= \exp \left( \int_0^t{\rm Re}h(\varphi(s,x))\right) f(\varphi(x,t)).\]

It holds that $\mathcal T^{\C}_{h}$ is stable on $L^p_\rho(\Omega, \C)$  if and only if $\mathcal T^{\R}_{{\rm Re}h}$ is stable on $L^p_\rho(\Omega, \R)$.

Indeed,  if $f\in L^p_\rho(\Omega, \R)$, then 
\[|T^{\C}_{h}(t)(f)(x)|=|T^{\R}_{{\rm Re}h}(t)(f)(x)|\]
and we get immediately that if $\mathcal T^{\C}_{h}$ is stable then $\mathcal T^{\R}_{{\rm Re}h}$ is stable.
Conversely, if $f\in L^p_\rho(\Omega, \C)$, then $|f|\in L^p_\rho(\Omega, \R)$ and 
\[ |T^{\C}_{h}(t)(f)(x)|=|T^{\R}_{{\rm Re} h}(t)(|f|)(x)|\]
and  again we get that if $\mathcal T^{\R}_{{\rm Re}h}$ is stable, then $ \mathcal T^{\C}_{h}$ is stable too.

Taking into account the previous consideration, we will only consider the real valued function space $L^p_\rho(\Omega,\R)$.}
\end{remark}

\section{One-dimensional case on Lebesgue spaces}

In the following  sections  we assume that:

\begin{itemize}
 \item[(H1)] $\Omega\subseteq \R$ open, $h\in C(\Omega,\R)$;
 \item[(H2)] $F\in C^1(\Omega,\R)$ and  $\varphi$ is the semiflow associated with $F$;
  i.e. for every $x_0\in\Omega$, $\varphi(\cdot,x_0): J(x_0) \rightarrow \R$  is the unique solution of the initial value problem
\[\dot{x}=F(x),\; x(0)=x_0\]
where the  $J(x_0)\subseteq\R$ is the maximal domain  of $\varphi( \cdot, x_0)$; It is known that $J(x_0)$ is an open interval such that $0\in J(x_0)$;  
\item[(H3)] $[0,\infty)\subset J(x_0)$ for every $x_0\in\Omega$. 
\end{itemize}

\noindent We refer to the monograph of Amann \cite{amann1990an} for further results on the topic of flows.
Let us define the following subsets of $\Omega$:
    \[ \Omega_0:=\{x\in \Omega\,\mid\, F(x)=0\},\qquad \Omega_1=\Omega\setminus\Omega_0.\]

To emphasize the dependence on $F$ and $h$, we will denote the semigroup defined  in \eqref{semi} by  $\mathcal T_{F,h}=(T_{F,h}(t))_{t\geq 0}$.

We recall the following results that were proved in \cite{kalmes2009hypercyclic,aroza_kalmes_mangino2013chaotic}, Lemma 7 and Corollary 12 respectively.

 \begin{lemma}[\cite{aroza_kalmes_mangino2013chaotic}]\label{crucial lemma}
Let $\Omega\subseteq\R, F,h$ satisfying (H1)-(H3) and let $\rho$ be $p$-admissible measurable function for $F$ and $h$.
 \begin{enumerate}
 \item Let $[a,b]\subset\Omega_1$ and set  $\alpha:=a,\beta:=b$ if $F_{|[a,b]}>0$,
respectively $\alpha:=b,\beta:=a$ if $F_{|[a,b]}<0$. There is a constant $C>0$ such that
\[\forall\,x\in[a,b]:\;\frac{1}{C}\leq\rho(x)\leq C\]
and
\[\forall\,t\geq 0, x\in[a,b]:\;\frac{1}{C}\rho_{t,p}(\alpha)\leq\rho_{t,p}(x)\leq C\rho_{t,p}(\beta).\]

\item For all $t\geq 0$ and $x\in\Omega$

\begin{eqnarray*}
\rho_{t,p}(x)&=&\chi_{\varphi(t,\Omega)}(x)\exp\left(p\int_{-t}^0\left[h(\varphi(s,x))-\frac{1}{p}F'(\varphi(s,x))\right] ds\right)\rho(\varphi(-t,x))\\\\
&=&\begin{cases}\exp\left(pt [h(x)-\frac{1}{p}F'(x)]\right)\rho(x),&x\in \Omega_0,\\\\
		\chi_{\varphi(t,\Omega)}(x)\exp\left(p\displaystyle \int_{\varphi(-t,x)}^x\dfrac{h(y)-\frac{1}{p}F'(y)}{F(y)} dy\right)\rho(\varphi(-t,x)),&x\in \Omega_1\end{cases}  
\end{eqnarray*}
\end{enumerate}
\end{lemma}

\begin{remark}\label{F'zero}
	{\rm Observe that if $\lambda(\Omega_0)>0$ then 
	$$\rho_{t,p}(x)=e^{ pt h(x)} \rho(x),\qquad \lambda-\mbox{a.e. } x\in\Omega_0.$$
	In fact, we have that  $\Omega_0=\bigcup_{n\in \N} (a_n,b_n)$, where $a_n<b_n$; being $F\in C^1$ with $F(x)=0$ on $(a_n, b_n)$,  then $F'=0$ in $(a_n, b_n)$ for every $n\in\N$, thus $F'=0$ in $\Omega_0$.}
\end{remark}

A first consequence of Theorem~\ref{characterization stability_1} is a characterization of stability for the $C_0$-semigroup $\mathcal{T}_{F,h}$ on $L^p_\rho(\Omega)$.

\begin{theorem}\label{characterization stability_2}
Let $\Omega\subseteq\R, F,h$ satisfying (H1)-(H3), and let $\rho$ be $p$-admissible measurable function for $F$ and $h$.
Then the following conditions are equivalent:
\begin{itemize}
	\item[i)] $\mathcal{T}_{F,h}$ is stable on $L^p_\rho (\Omega)$,
	\item[ii)] It holds:
\begin{enumerate}
\item[(1)]  $\mathcal{T}_{F,h}$ is bounded;
\item[(2)]  $\lim_{t\to+\infty} \rho_{t,p}(x)=0$ for $\lambda$-a.e. $x\in \Omega_1$
\item[(3)]  if $\lambda(\Omega_0)>0$, $h(x)<0$ $\lambda$-a.e. in $\Omega_0$.
\end{enumerate}
\end{itemize}

\end{theorem}

\begin{proof} 
 If we define 
 \[ X_i=\{ f\in L^p_\rho(\Omega)\,\mid\, f=0\ \mbox{a.e. in}\  \Omega_i\},\qquad i=0,1,\]
 clearly we can identify $X_i$ with $L^p_\rho(\Omega_i)$ and 
 \[ L^p_\rho(\Omega)=X_0 \oplus X_1= L^p_\rho(\Omega_0)\oplus L^p_\rho(\Omega_1).\]
 If $\lambda(\Omega_0)=0$, then $X_0$ reduces to $\{0\}$ and $L^p_\rho(\Omega)$ can be identified with $L^p_\rho(\Omega_1)$.
  
 For every $x_0\in\Omega_0$ we have  that $\varphi(t,x_0)=x_0$ for every $t\geq 0$. By the  uniqueness of  the solutions of the initial value problems
\[\dot{x}=F(x),\; x(0)=x_0\, (x_0\in\Omega),\]
$\varphi(t, \Omega_0)\subseteq \Omega_0$ and $\varphi(t, \Omega_1)\subseteq \Omega_1$.
This implies that $L^p_\rho(\Omega_i)$ is invariant  under $\mathcal{T}_{F,h}$ for $i=0,1$. 
Thus we can define $\mathcal{T}^i_{F,h}= {\mathcal{T}_{F,h}}_{|L_\rho^p(\Omega_i)}$ and we can write 
\[ \mathcal{T}_{F,h} = \mathcal{T}^0_{F,h}\oplus \mathcal{T}^1_{F,h}.\]
Clearly $\mathcal{T}_{F,h}$ is stable on $L^p_\rho(\Omega)$ if and only if $\mathcal{T}^i_{F,h}$, $i=0,1$, are stable on $L^p_\rho(\Omega_i)$.

Observe moreover that $T^0_{F,h}(t)(f)(x)=\exp(th(x))f(x)$ for every $f\in L^p_\rho(\Omega_0)$ and $x\in\Omega_0$.

$``\Rightarrow''$: Let $x\in\Omega_1$ such that $F(x)\not=0$. If $F(x)>0$ there exists $r>0$ such that $[x, x+r]\subseteq \Omega_1$ with $F(s)>0$ for $s\in [x, x+r]$.  By Lemma \ref{crucial lemma}, there exists $C>0$ such that 
\[\rho_{t, p}(x)\leq C\rho_{t,p}(s) \qquad \mbox{a.e. } s\in [x, x+r],\]
hence
\[\rho_{t, p}(x)= \frac 1 r \int_x^{x+r} \rho_{t,p}(x)ds \leq \frac{C}{r}\int _x^{x+r} \rho_{t,p}(s)ds ,\]
and therefore, by assumption and Theorem~\ref{characterization stability_1}, $\lim_{t\to\infty} \rho_{t, p}(x)=0$.
If $F(x)<0$,  we consider an interval $[x-r, x]$ with $F(s)<0$ for $s\in [x-r, x]$  and we get the assertion again by Lemma \ref{crucial lemma} arguing as in the case $F(x)>0$.

By the stability of $\mathcal{T}^0_{F,h}$,  we get that
\[ \lim_{t\to\infty}\int_{\Omega_0} e^{pth(x)}|f(x)|^p\rho(x)dx=0\]
hence either $\lambda(\Omega_0)=0$  or if $\lambda(\Omega_0)>0$ we have $h(x)<0$ $\lambda$ -a.e. in $\Omega_0$.


$''\Leftarrow''$: By (1) ,  there exists $M>0$ such that $\rho_{t,p}(s)\leq M\rho(s)$ a.e. in $\Omega$. Then, for every $f\in L^p_\rho(\Omega_1)$, being $\rho$ locally integrable on $\Omega$, we can apply the dominated convergence theorem to get that for any bounded interval $Q\subseteq \Omega$:
\[  \lim_{t\to\infty} \int_{Q} \rho_{t,p}(s)ds =  0. \]
\end{proof}

\begin{example}[Left translation semigroup]
	{\rm Let $\Omega=\R, F=1,$ and $h=0$.  It is easily seen that $\rho$ is $p$-admissible for $F$ and $h$ for some $p\in[1,\infty)$ if the same holds for every $p\in[1,\infty)$ and   the $C_0$-semigroup $\mathcal{T}_{F,h}$ on $L^p_\rho(\R)$ is defined by $(T_{F,h}(t)f)(x)=f(x+t)$. Moreover, we have $\Omega_0=\emptyset$ and $\rho_{t,p}(x)=\rho(x-t)$.
	
	By Theorem~\ref{characterization stability_2}, $\mathcal{T}_{F,h}$ is stable on $L^p_\rho(\R)$ if and only if for every $x\in \R$ there exists a real finite constant $C\geq 0$ such that
	\[\rho(x-t)\leq C\rho(x),\ \forall t \geq 0, x\in\R \mbox{ and }\lim_{x\to -\infty}\rho(x)=0.\]
	This condition is independent of $p$.}
	
\end{example}

\begin{example}
	{\rm Let $\Omega=\R$, $F(x):=1-x,\ h(x)=0$ so that $h_t(x)=1$, $\varphi(t,\R)=\R$ and $\partial_2\varphi(t,x)=e^{-t}$, where $\varphi(t,x)=1+(x-1)e^{-t}$. 
	In this case $\mathcal{T}_{F,h}$ is given by $(T_{F,h}(t)f)(x)=f(1+(x-1)e^{-t})$.
	Furthermore, $\Omega_0=\{1\}$, thus  $\lambda(\Omega_0)=0$, and
	\[\forall\,t\geq 0: \;\rho_{t,p}(x)=\rho(1+(x-1)e^t)e^t,\quad x\in \Omega_1=\R\setminus\{-1\}.\]
	By Theorem~\ref{characterization stability_2}  $\mathcal{T}_{F,h}$ is stable on $L^p_\rho(\R)$ if and only if  there exists a real finite constant $C\geq 0$ such that
	\[\rho(1+(x-1)e^t)e^t\leq C\rho(x),\ \forall t\geq 0,\ \forall x\in \R\setminus\{-1\}\mbox{ and }\lim_{|r|\to \infty}\rho(r)r=0.\]}
\end{example}

We can simplify the assumptions of Theorem~\ref{characterization stability_2}, if  we have more information about the flow:

\begin{corollary} Let $\Omega\subseteq\R,F,h$ satisfying (H1)-(H3) and let $\rho$ be $p$-admissible measurable function for $F$ and $h$. Assume that 
\begin{equation*} \forall x\in \Omega_1,\ \exists \ \overline t>0 \ : \qquad x\notin \varphi(\overline t, \Omega). \end{equation*}
 where $\varphi$ is the semiflow associated with $F$. Then the following conditions are equivalent:
\begin{itemize}
	\item[(i)] $T_{F,h}$ is stable on $L^p_\rho (\Omega)$,
	\item[(ii)] It holds:
\begin{enumerate}
\item[(1)] $T_{F,h}$ is bounded, 
\item[(2)] if $\lambda(\Omega_0)>0$, $h(x)<0$ $\lambda$-a.e. in $\Omega_0$.
\end{enumerate}
\end{itemize}
\end{corollary}

\begin{proof}  Simply observe that the assumption implies that  $x\notin \varphi(t,\Omega)$ for every $t>\overline t$ and for every $x\in\Omega_1$, and therefore $\lim_{t\to\infty}\rho_{t,p}(x)=0$.
\end{proof}

We will apply this corollary to the von Foerster- Lasota semigroup in Section 4.

If  $\rho=1$, a straightforward consequence of Lemma \ref{crucial lemma}(2) gives the following characterization.

\begin{theorem}\label{stability_rho1}
	Let $\Omega\subseteq\R, F,h$ satisfying (H1)-(H3) and assume that $\rho=1$ is a $p$-admissible function for $F$ and $h$. Assume that $F(x)\not=0$ for every $x\in \Omega$.
\begin{enumerate}
\item[(1)] If $\varphi(t, \Omega)=\Omega$ for every $t>0$, 
then t.f.a.e.:
\begin{itemize}
	\item[(i)] $\mathcal{T}_{F,h}$ is stable on $L^p(\Omega)$;
	\item[(ii)] It holds, 
					\begin{itemize}
						\item[(a)] there exists $C\in\R$ such that 
							\begin{equation*}\label{FF'} \int_y^{\varphi(t,y)}\frac{h(s)-\frac 1 p F'(s)}{F(s)}ds\leq C \ \ \ a.e.\ y\in \Omega, \ t\geq 0,\end{equation*}
							or, equivalently, $\mathcal{T}_{F,h}$ is bounded;
						\item[(b)] for every $y\in\Omega$
							\[ \lim_{t\to +\infty} \int_y^{\varphi(t,y)}\frac{h(s)-\frac 1 p F'(s)}{F(s)}ds =-\infty.\]
					\end{itemize}
\end{itemize}
\item[(2)] If  
\begin{equation*} \forall x\in \Omega_1\ \exists \ \overline t>0 \qquad x\notin \varphi(\overline t, \Omega). \end{equation*}
then t.f.a.e.:
\begin{itemize} \item[(i)] $\mathcal{T}_{F,h}$ is stable on $L^p(\Omega)$;
\item[(ii)]  there exists $C\in\R$ such that 
\begin{equation*}\int_y^{\varphi(t,y)}\frac{h(s)-\frac 1 p F'(s)}{F(s)}ds\leq C \ \ \ a.e.\ y\in \Omega, \ t\geq 0,\end{equation*}
or, equivalently, $\mathcal{T}_{F,h}$ is bounded.
\end{itemize}
\end{enumerate}
\end{theorem}

\section{Stability on Sobolev spaces}\label{sobolev_one_dimensional_2}

Throughout  this section, let $I=(a,b)$ be a bounded open interval of $\R$. For $1\leq p\leq \infty$ we set as usual 
\[W^{1,p}(I)=\{u\in L^p(I); u'\in L^p(I)\},\]
where $u'$ denotes the distributional derivative of $u$. Endowed with the norm
\[\norm{u}_{1,p}=\norm{u}_p+\norm{u'}_p,\]
$W^{1,p}(I)$ is a Banach space. It holds that $W^{1,p}(I)\subseteq C [a,b]$ and that for any $x\in [a,b]$ the point evaluation $\delta_x$ in $x$ is a continuous linear form on $W^{1,p}(I)$. We are interested also in the following closed subspace of $W^{1,p}(I)$,
\[W^{1,p}_*(I):=\mbox{ker}\,\delta_a.\]
From the boundedness of $I$ we have the topological direct sum
\[W^{1,p}(I)=W^{1,p}_*(I)\oplus \mbox{span}\,\{\1\},\]
where $\1$ denotes the constant function with value $1$.

Let $F:[a,b]\rightarrow\R$ a $C^1$-function satisfying (H2)-(H3) with $F(a)=0$, $h\in C([a,b])$ and consider   the restriction $\mathcal S_{F,h}$ of $\mathcal T_{F,h}$ to $W^{1,p}(I)$, for every $1\leq p<\infty$. Then $\mathcal S_{F,h}$ is a $C_0$-semigroup on $W^{1,p}(I)$ and $W_* ^{1,p}(I)$ is invariant for $\mathcal S_{F,h}$ (see \cite[Proposistion 23]{aroza_kalmes_mangino2013chaotic}). 

We recall  that two $C_0$-semigroups $\mathcal{T}$ and $\mathcal{S}$ , on Banach spaces $Y$ and $X$ respectively, are said to be conjugate if there exists a linear homeomorphism  $\phi :Y\to X$ with dense range such that $\phi \circ T(t)=S(t)\circ \phi$, for every $t\geq 0$. It is immediate  that stability is invariant under conjugacy.

\begin{theorem}\label{stability_sobolevspaces}
	Let $I=(a,b)$ be a bounded interval, $1\leq p<\infty$, $F$ satisfying (H2), (H3) with $F(a)=0$ and $F(x)\not=0$ in $]a,b[$ and $h\in W^{1,\infty}(I)$.  Assume that
the function $[a,b]\rightarrow\R,y\mapsto\frac{h(y)-h(a)}{F(y)}$ belongs to $L^\infty(I)$. 	The following  conditions hold:
\begin{enumerate}
	\item[(1)] If $\varphi(t, I)=I$ for every $t>0$, 
		then t.f.a.e.:
		\begin{itemize}
			\item[(i)] $\mathcal{S}_{F,h}$ is stable on $W^{1,p}_*(I)$;
			\item[(ii)] It holds, 
					\begin{itemize}
						\item[(a)] there exists $C\in\R$ such that 
							\begin{equation*} \int_y^{\varphi(t,y)}\frac{h(a)-\left(\frac 1 p -1\right) F'(s)}{F(s)}ds\leq C \ \ \ a.e.\ y\in I, \ t\geq 0;\end{equation*}
						\item[(b)] for every $y\in I$
							\[ \lim_{t\to +\infty} \int_y^{\varphi(t,y)}\frac{h(a)-\left(\frac 1 p -1\right) F'(s)}{F(s)}ds =-\infty.\]
					\end{itemize}
		\end{itemize}
	\item[(2)] If  
		\begin{equation*} \forall x\in I,\ \exists \ \overline t>0\ : \quad x\notin \varphi(\overline t, I) \end{equation*}
		then t.f.a.e.:
		\begin{itemize} 
			\item[(i)] $\mathcal{S}_{F,h}$ is stable on $W^{1,p}_*(I)$;
			\item[(ii)]  there exists $C\in\R$ such that 
				\begin{equation*} \int_y^{\varphi(t,y)}\frac{h(a)-\left(\frac 1 p -1\right) F'(s)}{F(s)}ds\leq C \ \ \ a.e.\ y\in I, \ t\geq 0.\end{equation*}
		\end{itemize}
\end{enumerate}
Moreover, $\mathcal{S}_{F,h}$ is stable on $W^{1,p}(I)$ if and only if it is stable on $W^{1,p}_*(I)$ and $h(a)<0$.
\end{theorem}

\begin{proof}
	By  the discussion in \cite[Section 3]{aroza_kalmes_mangino2013chaotic}, 	the $C_0$-semigroup $\mathcal{S}_{F,h}$ on $W^{1,p}_*(I)$ is conjugate to $\mathcal{T}_{F,F'+h(a)}$ on $L^p(I)$. 	Thus, $\mathcal{S}_{F,h}$ is stable on $W^{1,p}_*(I)$ if and only if $\mathcal{T}_{F,F'+h(a)}$ is stable on $L^p(I)$.
	
	A short calculation shows that $\mathcal{T}_{F,F'+h(a)}$ is stable on $L^p(I)$ if and only if  condition (ii) holds. In fact, for $\mathcal{T}_{F,F'+h(a)}$ observe that $h_t(x)=e^{th(a)}\partial_2 \varphi(t,x)$ and $\rho(x)=1$ is p-admissible for $F$ and $F'+h(a)$ taking, for example, $M=1$ and $w=ph(a)+(p-1)\|F'\|_{\infty}$. Depending on the behavior of $\varphi(t,\Omega)$ for $t\geq 0$, the proof on $W^{1,p}_*(I)$ is over if we replace $h$ by $F'+h(a)$ in the conditions of Theorem~\ref{stability_rho1}.

Finally, observe that 
\[\mathcal{S}_{F,h(a)}=\mathcal{S}_{F,h(a)_{|W^{1,p}_*(I)}} \oplus \mathcal{S}_{F,h(a)_{|\mbox{span}\,\{\1\}}}\]

Thus $\mathcal{S}_{F,h(a)}$ is stable if and only if  $\mathcal{S}_{F,h(a)}$ is stable on $W^{1,p} _*(I)$ and on $\mbox{span}\,\{\1\}$, i.e., $h(a)<0$. In fact, for $\lambda \neq 0$ we have
\[\mathcal{S}_{F,h(a)}(\lambda \1)=e^{h(a)t}\lambda \stackrel{t\to \infty}{\rightarrow}0 \mbox{ if and only if } h(a)<0.\]
We get the assertion since, by the discussion in \cite[Section 3]{aroza_kalmes_mangino2013chaotic},  the semigroups $\mathcal{S}_{F,h(a)}$ and $\mathcal{S}_{F,h}$  are conjugate.

\end{proof}

\section{The von Foerster - Lasota equation}

\subsection{The linear von Foerster-Lasota equation}

Let $I=(0,1)$ and $h\in C([0,1])$.	 Consider the linear von Foerster-Lasota equation 
\begin{equation}\label{eq:lasota}\frac{\partial}{\partial\,t}u(t,x)+x\frac{\partial}{\partial\,x}u(t,x)=h(x)\,u(t,x),\;\;t\geq 0,\, 0<x< 1\end{equation}
with the initial condition
\[u(0,x)=v(x),\;\;0 <x< 1,\]
where $v$ is a given function.
This equation  is a particular case of the equation
\begin{equation*}\frac{\partial u}{\partial t}(t,x) + c(x) \frac{\partial u}{\partial x}(t,x)=f(x,u(t,x)) \qquad t\geq0,\ x\in [0,1] \end{equation*}
that was introduced in  \cite{lasota_mackey_wazewska1981_minimizing} to describe the reproduction of a population of red blood cells, mainly in connection with studies about anemia. 
Defining
\begin{equation} T_h(t)v(x)=\exp{\left(\int_{-t} ^0 h(xe^s)ds\right)} v(xe^{-t}),\qquad t\geq0,\ x\in [0,1], \label{sem:lasota}\end{equation}
the family $\mathcal{T}_h=(T_h(t))_{t\geq 0}$ is a $C_0$-semigroup in $L^p([0,1])$ with $1\leq p <\infty$, and $u(t,x)=T_h(t)v(x)$, $t\geq 0, x\in [0,1]$ is the  solution of the equation \eqref{eq:lasota} with initial value $v\in L^p([0,1])$.

Clearly $\mathcal T_h=\mathcal T_{F,h}$ with $F(x)=-x$.
 In  analogy with the  previous section, we denote by $\mathcal{S}_h$ the restrictions of $\mathcal T_h$ to $W^{1,p}(0,1)$.

After the paper \cite{lasota1981stable}, the asymptotic behaviour of $\mathcal T_h$ has  been studied  in different function spaces by several authors  (see e.g.  \cite{rudnicki2012chaoticity, brzezniak_dawidowicz2009on, brzezniak_dawidowicz2013on} and the references quoted therein). We recover their results about stability by applying the discussion of Section 2 and 3.

\begin{theorem}
$\ $
\begin{itemize}
  \item[a)] Assume that for $h\in C[0,1]$  the function 
  \[[0,1]\rightarrow\R,\qquad x\mapsto \frac{h(x)-h(0)}{x}\]
  belongs to $L^1(0,1)$. Then the following properties of the associated von Foerster-Lasota semigroup $\mathcal{T}_h$ on $L^p(0,1)$ are equivalent.
  \begin{itemize}
    \item[i)] $\mathcal{T}_h$ is stable on $L^p(0,1)$.
    \item[ii)] $h(0)\leq -\frac{1}{p}$.
  \end{itemize}
  \item[b)] Assume that for $h\in W^{1,\infty}(0,1)$  the function 
  \[[0,1]\rightarrow\R,\qquad x\mapsto \frac{h(x)-h(0)}{x}\]
  belongs to $L^\infty(0,1)$. Then, for the von Foerster-Lasota semigroup $\mathcal{S}_h$,
  t.f.a.e.:
		\begin{itemize} 
			\item[(i)] $\mathcal{S}_{F,h}$ is stable on $W^{1,p}_*(I)$;
			\item[(ii)] $h(0)\leq1-\frac{1}{p}$.
		\end{itemize}

Moreover, $\mathcal{S}_{h}$ is stable on $W^{1,p}(I)$ if and only if  $h(0)<0$.
\end{itemize}  

\end{theorem} 

\begin{proof}
For $F(x)=-x$ we have $\varphi(t,x)=xe^{-t}$ thus,  for every $x\in (0,1)$ we get that $x\notin \varphi(t, \Omega)$ if  $t>-\log{x}$ . 

Proof of part a).\\
For $\lambda$-a.e. $x\in I$ and for all $t\geq 0$,
\[\int_x ^{xe^{-t}} \frac{\,h(y)+\frac{1}{p}}{-y} dy =\int_{xe^{-t}} ^x\frac{\,h(y)-h(0)+h(0)+\frac{1}{p}}{y} dy.\]

Since \[[0,1]\rightarrow\R,\qquad x\mapsto \frac{h(x)-h(0)}{x}\] belongs to $L^1(0,1)$, we obtain for some constant $K\geq 0$ that
\[\int_{xe^{-t}} ^x\frac{\,h(y)-h(0)+h(0)+\frac{1}{p}}{y} dy \leq K + \int_{xe^{-t}} ^x\frac{h(0)+\frac{1}{p}}{y} dy.\]	

By Theorem~\ref{stability_rho1}(2), $\mathcal{T}_h$ is stable if and only if there exists  $C\in\R$ such that for $\lambda$-a.e. $x\in I$ and for all $t\geq 0$ 
\[\int_{xe^{-t}} ^x\frac{h(0)+\frac{1}{p}}{y} dy \leq C.\]

Observe that \[\int_{xe^{-t}} ^x\frac{h(0)+\frac{1}{p}}{y} dy=\left(h(0)+\frac{1}{p}\right)t.\]

Then $\mathcal{T}_h$ is stable if and only if $h(0)\leq -\frac{1}{p}$.

Proof of part b).\\
It follows with straightforward calculations from Theorem~\ref{stability_sobolevspaces}(2). \end{proof}

\subsection{Generalized von Foerster-Lasota equation}
 
	Let us consider $I=(0,1)$, $h\in C([0,1])$, $r>1$  and the first order partial differential equation
\[\frac{\partial}{\partial\,t}u(t,x)+x^r\frac{\partial}{\partial\,x}u(t,x)=h(x)\,u(t,x),\quad t\geq 0,\, 0<x< 1\]
with the initial condition
\[u(0,x)=v(x),\quad 0 <x< 1,\]
where $v$ is a given function. 
Denote by $\mathcal{T}_{r,h}$ the $C_0$-semigroup on $L^p(0,1)$ associated with $F(x)=-x^r$ and $h$ and by  $\mathcal{S}_{r,h}$ its restriction to $W^{1,p}(0,1)$. 

\begin{theorem}
$\ $
\begin{itemize}
  \item[a)] Let $h\in C[0,1]$  and assume that the function 
  \[[0,1]\rightarrow\R,\qquad x\mapsto \frac{h(x)-x^{r-1}h(0)}{x^r}\]
  belongs to $L^1(0,1)$. Then the following properties of $\mathcal{T}_{r,h}$ on $L^p(0,1)$ are equivalent.
  \begin{itemize}
    \item[i)] $\mathcal{T}_{r,h}$ is stable.
    \item[ii)] $h(0)\leq \frac{-r}{p}$.
  \end{itemize}
  \item[b)] Let $h\in W^{1,\infty}(0,1)$  and assume that the function 
  \[[0,1]\rightarrow\R,\qquad x\mapsto \frac{h(x)-h(0)}{x^r}\]
  belongs to $L^\infty(0,1)$. Then 
  the following are equivalent.
  \begin{itemize}
    \item[i)] $\mathcal{S}_{r,h}$ is stable on $W^{1,p}_*(0,1)$.
    \item[ii)] $h(0)\leq 0$.
  \end{itemize}
  Moreover, $\mathcal{S}_{h}$ is stable on $W^{1,p}(I)$ if and only if  $h(0)<0$.  
\end{itemize}
\end{theorem} 

\begin{proof}
Recall that $F(x)=-x^r$ thus,  $\varphi(t,x)=((r-1)t+x^{1-r})^\frac{1}{1-r}$ and
\[\forall\,t\geq 0, x\in (0,1]:\qquad \partial_2\varphi(t,x)=x^{-r}((r-1)t+x^{1-r})^{\frac{r}{1-r}}\]

 For all $x\in (0,1)$,  choosing $\overline t>\frac{x^{1-r}-1}{r-1}$ we have $x\notin \varphi(\overline t, (0,1))$. By applying Theorem~\ref{stability_rho1}(2) and Theorem~\ref{stability_sobolevspaces}(2) we get  the proofs of a) and b) respectively. In fact, for the proof of part a) since $[0,1]\rightarrow\R,\qquad x\mapsto \frac{h(x)-x^{r-1}h(0)}{x^r}$ belongs to $L^1(0,1)$, if we denote by $K$ its norm, we obtain for $y\in (0,1)$ and $t\geq 0$

\begin{align*}
		\int_y ^{\varphi(t,y)} \frac{h(s)-\frac{1}{p}F'(s)}{F(s)}ds= & \int_{\varphi(t,y)} ^y \frac{h(s)+\frac{r}{p}s^{r-1}}{s^r}ds\\
		= & \int_{\varphi(t,y)} ^y \frac{h(s)+s^{r-1}h(0)}{s^r}ds+\int_{\varphi(t,y)} ^y \frac{h(0)+\frac{r}{p}}{s}ds\\
		\leq & K +\int_{\varphi(t,y)} ^y \frac{h(0)+\frac{r}{p}}{s}ds\\
		= & K+ \left(h(0)+\frac{r}{p}\right)\log{\left(\frac{y}{\varphi(t,y)}\right)}\\
		= & K+ \frac{1}{r-1}\left(h(0)+\frac{r}{p}\right)\log{\left(1+(r-1)ty^{r-1}\right)}
\end{align*}

Then $\mathcal{T}_{r,h}$ is stable on $L^p(0,1)$ if and only if $\left(h(0)+\frac{r}{p}\right)\leq 0$, by Theorem~\ref{stability_rho1}(2).

Concerning part b), by using Theorem~\ref{stability_sobolevspaces}(2) and by observing that, since $p\geq 1$, $r>1$, and $\frac{y}{\varphi(t,y)}=(1+(r-1)ty^{r-1})^{\frac{1}{r-1}}$,  we have for some $C>0$ 
\[\left(\frac 1p -1\right)r\log\left(\frac{y}{\varphi(t,y)}\right) \leq 0 ,\]
we get that 
\begin{align*}
		\int_y^{\varphi(t,y)}\frac{h(0)-\left(\frac 1 p -1\right) F'(s)}{F(s)}ds \leq & -\frac{h(0)}{r-1}\left(y^{1-r} - \varphi(t,y)^{1-r}\right)+ C \\
		= & h(0)t + C.
\end{align*}

Finally, $S_{r,h}$ is stable on $W^{1,p} _* (0,1)$ if and only if $h(0)\leq 0$, by Theorem~\ref{stability_sobolevspaces}(2). Of course, $S_{r,h}$ is stable in $W^{1,p}$ if and only if $h(0)< 0$ by the same Theorem.

\end{proof}

\section*{Appendix}

We would like now to compare the previous stability conditions with hypercyclicity conditions.  We recall that a $C_0$-semigroup $\mathcal T=(T_t)_{t\geq0}$ on a separable Banach space $X$,
  is said to be hypercyclic if  there exist $x\in X$, called hypercyclic vector, such that its orbit $\{T(t)x\ :\ t\geq 0\}$ is dense in $X$. We refer to the recent monographs \cite{bayart_matheron2009dynamics, grosse-erdmann_peris2011linear} for a complete reference on this topic.
The hypercyclic behaviour of the weighted composition semigroups $\mathcal T_{F,h}$ has been characterized by Kalmes in \cite{kalmes2007hypercyclic} (see also \cite{kalmes2016asimple, kalmes2015aremark} for further hypercylicity related conditions on the same semigroups):

\begin{theorem}[\cite{kalmes2007hypercyclic}]\label{hyp}
Let $\Omega\subseteq\R$ be open, $F, h$ satisfying (H1)-(H3),  $\rho:\Omega\rightarrow (0,\infty)$ be a measurable function which is $p$-admissible for $F$ and $h$.
For the $C_0$-semigroup $\mathcal{T}_{F,h}$ on $L^p_\rho(\Omega)$ the following are equivalent.
	\begin{itemize}
		\item[i)] $\mathcal{T}_{F,h}$ is hypercyclic.
		
		\item[ii)] $\lambda(\Omega_0)=0$ and for every $m\in\N$ for which there are $m$ different connected components $C_1,\ldots,C_m$ of $\Omega_1$, for $\lambda^m$-almost all choices of $(x_1,\ldots,x_m)\in\Pi_{j=1}^m C_j$ there is a sequence of positive numbers $(t_n)_{n\in\N}$ tending to infinity such that
		\[\forall\,1\leq j \leq m:\;\lim_{n\rightarrow\infty}\rho_{t_n,p}(x_j)=\lim_{n\rightarrow\infty}\rho_{-t_n,p}(x_j)=0, \]
	where $\rho_{-t,p}:\Omega\rightarrow [0,\infty),\qquad \rho_{-t,p}(x):=|h_t(x)|^{-p}\rho(\varphi(t,x))\partial_2\varphi(t,x).$
	\end{itemize}
	
\end{theorem}

Clearly if a semigroup is hypercyclic, then it cannot be stable and, in general, a semigroup can be not stable and not hypercyclic. Indeed, if we consider a semigroup $T_{F,h}$ such that $\lambda(\Omega_0)>0$ and $h(x)>0$ on $\Omega_0$ then we get a semigroup which is not stable by Theorem~\ref{characterization stability_2} and not hypercyclic by Theorem \ref{hyp}.
Nevertheless, there are cases in which stability and not hypercyclicity are equivalent.  This happens, under suitable assumptions,  e.g. for the von Foerster Lasota semigroup,  as it has been proved by Dawidowicz and Poskrobko in \cite{dawidowicz_poskrobko2008onperiodic}.

This result will be covered by the next theorem, in the particular case $F=-x$.  

\begin{theorem}\label{nothyp_stable}
Let $\Omega=(\alpha, \beta) \subseteq\R$ be a bounded interval, $F\in C^1([\alpha,\beta])$ satisfying (H2)-(H3). Assume $F$ decreasing, $F(x)<0$ for each $x\in (\alpha,\beta]$, $F(\alpha)=0$ and such that 
 \[ \forall x\in \Omega_1,\ \exists \ \overline t>0 \ : \qquad x\notin \varphi(\overline t, \Omega).\]
  Moreover, let  $h=-\lambda F'$ for some $\lambda\in\R$.

 Then, for the $C_0$-semigroup $\mathcal{T}_{F,h}$ on $L^p(\Omega)$ the following are equivalent.
  \begin{itemize}
    \item[(i)] $\mathcal{T}_{F,h}$ is  stable.
    \item[(ii)] $\mathcal{T}_{F,h}$ is not hypercyclic.
    \item[(iii)] $\lambda\leq -\frac 1 p$
    \end{itemize}
\end{theorem}

  \begin{proof} Observe that $\rho=1$ is p-admissible for $F$ and $h$ by \cite[Lemma 19]{aroza_kalmes_mangino2013chaotic}. If $\lambda=-\frac 1 p$, then clearly the semigroup is stable and not hypercyclic. Then assume  that $\lambda\not=-\frac 1 p$.
  
  First observe that for every $y,z\in (\alpha, \beta)$ 
  \[ \int_y^z \frac{h(s)-\frac 1 pF'(s)}{F(s)}ds=-\left(\lambda + \frac 1 p\right)   \log\frac{F(z)}{F(y)}.\]
  In particular,  it follows that for every $z\in (\alpha,\beta)$ 
  \[\exists\  \lim_{y\to\alpha}  \int_y^z \frac{h(s)-\frac 1 pF'(s)}{F(s)}ds = -\left(\lambda + \frac 1 p\right) (+\infty)\]
  thus $\mathcal{T}_{F,h}$ is hypercyclic if and only if $\lambda>-\frac 1 p$.
  
  On the other hand, $\mathcal{T}_{F,h}$ is stable if and only if there exists $C>0$ such that 
  \[ \forall x\in (\alpha,\beta),\ \forall t>0:\ \ \int_{x}^{\varphi(t,x)}  \frac{h(s)-\frac 1 pF'(s)}{F(s)}ds\leq C,\]
  that is 
  \[ \forall x\in (\alpha,\beta),\ \forall t>0:\ \ -\left(\lambda + \frac 1 p\right)   \log\frac{F(\varphi(t,x))}{F(x)}\leq C.\]
  Being $F<0$ in $(\alpha,\beta)$, we have  that $\varphi(\cdot,x)$ is strictly decreasing, hence $\varphi(t,x)<x$ for every $t>0$. Hence, since $F$ is decreasing,
 \[ \forall x\in(\alpha,\beta), \ t >0:\ \  \frac{F(\varphi(t,x))}{F(x)} \leq 1.\]
 Thus, if $\lambda< -\frac 1 p$, then
  \[ \forall x\in (\alpha,\beta),\ \forall t>0:\ \ -\left(\lambda + \frac 1 p\right)   \log\frac{F(\varphi(t,x))}{F(x)}\leq 0,\]
  and so $\mathcal{T}_{F,h}$ is stable.
  \end{proof}
  
	\begin{remark}
		{\rm Of course it is possible to characterize stability for $\mathcal{T}_{F,h}$ if $F$ is strictly positive, taking into account that $\varphi$ would be increasing. In this case, we only need to change condition (iii) by $\lambda\geq -\frac 1 p$.}
	\end{remark}

\end{document}